\documentclass[12pt]{amsart}

\usepackage[all]{xy}
\usepackage{url}
\usepackage{amsmath}
\usepackage[margin=1.5 in]{geometry}
\usepackage{amssymb}
\usepackage{hyperref}
\usepackage{geometry}
\usepackage{color,soul}
\usepackage{cancel}
\usepackage[normalem]{ulem}

\newtheorem{theorem}{Theorem}[section]
\newtheorem{lemma}[theorem]{Lemma}
\newtheorem{proposition}[theorem]{Proposition}
\newtheorem{corollary}[theorem]{Corollary}

\theoremstyle{definition}
\newtheorem{definition}[theorem]{Definition}

\theoremstyle{remark}
\newtheorem{remark}[theorem]{Remark}

\numberwithin{equation}{section}

\usepackage{tikz-cd}
\usepackage{hyperref}

\title{Excellence in  Prime Characteristic}
\author{Rankeya Datta}
\address{Department of Mathematics, University of Michigan, 530 Church Street, Ann Arbor, MI 48109}
\email{rankeya@umich.edu}
\thanks{{The first author was partially supported by a Department of Mathematics Graduate Fellowship at the University of Michigan, and the second author's NSF grant DMS-1501625.}}

\author{Karen E. Smith}
\address{Department of Mathematics, University of Michigan, 530 Church Street, Ann Arbor, MI 48109}
\email{kesmith@umich.edu}
\thanks{{The second author was partially supported by NSF grant DMS-1501625.}}

\begin{document}

\maketitle

\centerline{\it{To Professor Lawrence Ein on the occasion of his sixtieth birthday.}}

\begin{abstract}  Fix  any field $K$ of characteristic $p$ such that $[K:K^p]$ is finite.
We discuss excellence for domains whose fraction field is $K$, showing for example, that $R$ is excellent if and only if the Frobenius map is finite on $R$. 
Furthermore, we show $R$ is excellent if and only if it admits  some 
non-zero $p^{-e}$-linear map for $R$ (in the language of \cite{BlicBock}), or equivalently, that the Frobenius map $R \rightarrow F_*R$ defines  a {\bf solid} $R$-algebra structure on $F_*R$  (in the language of \cite{HochSolid}).
In particular, this means that generically $F$-finite, Frobenius split Noetherian domains are always excellent.
We also show that non-excellent rings are abundant and easy to construct {in prime characteristic}, even within the world of regular local rings of dimension one inside function fields.
\end{abstract}

\section{Introduction}

The notion of {\bf excellence} for a commutative ring was introduced by Grothe-
-ndieck. A Noetherian ring is excellent, essentially, if the ring satisfies a list of axioms that ensures it  behaves much like a finitely generated algebra over a field; see Definition \ref{definition of excellence}.  An arbitrary Noetherian ring can be quite pathological, but the class of excellent rings is supposed to be the most  general setting to which one can expect the deep ideas of algebraic geometry, such as resolution of singularities, to  extend.

 Although a common hypothesis in the literature,  excellent rings are not widely understood. They  are often dismissed with sentences like  the following quoted from Wikipedia:  {\it ``Essentially all Noetherian rings that occur naturally in algebraic geometry or number theory are excellent; in fact it is quite hard to construct examples of Noetherian rings that are not excellent''} \cite{wiki:exc}.
 In this paper we show  that on the contrary,  non-excellent rings are quite easy to construct and are abundant, even among regular local rings of dimension one.  Our setting  is prime characteristic since Dedekind domains of characteristic zero are always excellent \cite[Cor 8.2.40]{Liu06}. The examples we construct, moreover, are {\it generically F-finite}, unlike other known examples such as {Nagata's} $k\otimes_{k^p}{k^p}[[t]]$ (whenever $[k:k^p] = \infty$),  whose completion map is purely inseparable. 
 
Excellence in prime characteristic is closely connected to another common hypothesis, that of {\it F-finiteness.} A ring of characteristic $p$ is $F$-finite if the Frobenius (or $p$-th power) map is finite. A well-known theorem of Kunz ensures that $F$-finite rings of characteristic $p$ are excellent \cite[Thm 2.5]{Kunz76}. The converse is also true under the additional hypothesis that the fraction field is $F$-finite. Put differently, a domain is $F$-finite if and only if it is excellent and generically $F$-finite.  For example, for any  domain $R$ whose fraction field is the function field of some algebraic variety over a perfect (e.g. algebraically closed) field, $R$ is $F$-finite if and only if $R$ is excellent. Because this fact does not seem to be well-known, we
show in Section 2 how this statement follows easily from  known facts in the literature. 
 
 In Section 3, we turn toward the issue of $p^{-e}$-linear maps.  
 For a ring $R$ of prime characteristic $p$, a $p^{-e}$-linear map is a map $R\overset{\phi}\rightarrow R$ of the underlying abelian group that satisfies  $\phi(r^{p^e}s) =  r\phi(s)$ for all $r, s \in S$. Any splitting of the Frobenius map is such a $p^{-1}$-linear map. In Section 3, we show that for {Noetherian} domains with F-finite fraction field, there are no non-zero $p^{-e}$-linear maps at all unless $R$ is excellent! Put differently using the language of Hochster \cite{HochSolid}, we show that {a generically $F$-finite Noetherian domain is a solid algebra via Frobenius if and only if it is excellent.   In particular, if a generically $F$-finite Noetherian domain is Frobenius split, it must be excellent.}

The study of  $p^{-e}$-linear maps, or equivalently, elements of   ${\text{Hom}}_R(F_*^eR,R)$, was formalized by Manuel Blickle and later used by Karl Schwede to give an alternate and more global approach to {\it  test ideals,} an important topic in characteristic $p$ commutative algebra. Our results show that Schwede's approach to test ideals relies much more heavily {on} $F$-finiteness  than previously understood. Test ideals can be viewed as ``prime characteristic analogs" of multiplier ideals due to the results in \cite{Smi00} and \cite{Har01} (see also \cite{HarYos03} and \cite{HarTag04}). While they have attracted great interest in birational algebraic geometry, our results in Section 4 offer a cautionary tale about the limits of this approach. 

In Section 4, we consider excellence in the setting of {\it discrete valuation rings} of a function field of characteristic $p$. Excellence in this case is equivalent to the DVR  being {\it divisorial,} a topic explored in \cite{DatSmi16}. We show here how this makes it is easy to write down explicit examples of non-excellent discrete valuation rings. Moreover, a simple countability argument shows that  among domains whose fraction field is, say the function field of $\mathbb P^2$, {\bf non-excellent} regular local rings of dimension one are far more abundant than the excellent ones. 

 This paper is largely expository,  drawing heavily on the work in \cite{DatSmi16} (and the corrections in \cite{DatSmi17}) where most of the harder proofs of the results discussed here appear. However, we are emphasizing a different aspect of the subject than in that paper, drawing conclusions not explicit there.
  
\noindent \textbf{Acknowledgments:} We are honored to help celebrate the birthday of Lawrence Ein, who has been a tremendous inspiration and support for the second author, both mathematically and personally. {We thank the referee for their careful comments, and, in particular, for suggesting a generalization of an earlier version of Proposition \ref{prop2}.}

\section{Excellence}
An arbitrary Noetherian ring can exhibit pathological behavior. For instance, the integral closure of a Noetherian domain in a finite extension of its fraction field can fail to be Noetherian, and Noetherian rings can have saturated chains of prime ideals of different lengths. Excellent rings were introduced by Grothendieck in  \cite[expos\'e IV]{Groth} to rule out such pathologies. We review Grothendieck's definition and some other relevant properties of excellent rings. Another good source is \cite[Chapter 32]{Mat80}.

\smallskip

\begin{definition}
\label{definition of excellence}
\cite[D\'ef 7.8.2]{Groth} A Noetherian ring $A$ is \textbf{excellent} if it satisfies the following properties:
\begin{enumerate}
\item $A$ is \emph{universally catenary}. This means that every finitely generated $A$-algebra has the  property that  
for any two prime ideals $p \subseteq q$, all saturated chains of prime ideals from $p$ to $q$ have the same length.\footnote{{The first example of a non-catenary Noetherian ring was given by Nagata in \cite{Nag56}.}}
\item   All {\it formal fibers } of $A$ are {\it geometrically regular. } This means that for every $p \in \operatorname{Spec}(A)$, the  fibers  of the natural map $\operatorname{Spec}(\widehat{A_p}) \rightarrow \operatorname{Spec}(A_p)$ induced by completion along $p$ are geometrically regular in the sense that for each  $x \in \operatorname{Spec}(A_p)$, the ring   $\widehat{A_p} \otimes_{A_p} K$ is regular for any finite field extension $K$ of the residue field $\kappa(x)$.

\item For every finitely generated $A$ algebra $B$, the regular locus of $ \operatorname{Spec}(B)$ is open; that is, the set $$\{q \in \operatorname{Spec}(B)\, : B_q {\text{ is a regular local ring}}\}$$ is open in $\operatorname{Spec}(B)$.
\end{enumerate} 
\end{definition}

The class of excellent rings is closed under homomorphic image, localization, and finite-type algebra extensions. Since every field is excellent, it follows that nearly every ring one is likely to encounter in classical algebraic geometry is excellent. Likewise, because the ring of integers is excellent and all complete local rings are excellent, familiar rings   likely to arise in number theory are excellent as well. 
 
{Among the many properties of excellent rings, the following, sometimes called the {\textbf{Japanese}} or {\textbf{N2}} property, will be important for us later.} 


\begin{proposition} 
\label{Japanese}  
\cite[expos\'e IV, 7.8.3 (vi)]{Groth}.  
Let $A$ be a Noetherian excellent domain.
  The integral closure of $A$ in any finite extension of its fraction field is {\it finite} as an $A$-module. 
\end{proposition}

{We construct discrete valuation rings of characteristic $p$ which fail to be Japanese in Subsection \ref{examples}.}

\medskip
\subsection{Excellence in prime characteristic}
Fix a  commutative ring  $R$ of prime characteristic $p$. 
The Frobenius map is the ring homomorphism $R \overset{F}\to R $ sending each element to its $p$-th power. 
It is convenient to denote the target copy of $R$ by $F_*R$.  Thus the notation  $F_*R$ denotes the ring $R$ but viewed as  an $R$-module  via the Frobenius map: an element  $r \in R$  acts on $x \in F_*R$ to produce $ r^px$.  Similarly, iterating the Frobenius map,  $F_*^eR$ denotes the $R$-algebra defined by the $e$-th iterate of Frobenius
 $R \overset{F^e}\longrightarrow F_*^eR$ sending $r\mapsto r^{p^e}$.

\begin{definition}
\label{definition of F-finite}
The ring $R$ is \textbf{F-finite} if the Frobenius map is finite; that is, $R$ is F-finite if $F_*R$ is a finitely generated $R$-module.
\end{definition}

F-finite rings are ubiquitous. For example,  it is easy to check that every perfect field is F-finite, and that a finitely generated algebra over an F-finite ring is F-finite. Furthermore, F-finiteness is preserved under homomorphic images, localization and completion, similar to excellence. Indeed, the two notions are closely related:

\begin{theorem} 
\label{ExcF-finite}
Let $R$ be a Noetherian domain whose fraction field $K$ satisfies $[K:K^p] < \infty$. Then $R$ is excellent if and only if $R$ is F-finite.
\end{theorem}

\begin{proof}
One direction of Theorem \ref{ExcF-finite} is the following famous {result} of Kunz:
\begin{theorem} \cite[Thm 2.5]{Kunz76}
\label{Kunz-excellent}
Let $R$ be any Noetherian ring of prime characteristic. If the Frobenius map $R\rightarrow F_*R$ is finite, then $R$ is excellent.
\end{theorem}
 
So to prove Theorem \ref{ExcF-finite}, we only need to prove the converse under the hypothesis that $R$ is a domain with F-finite fraction field. 
{The ring $R^p$, with fraction field $K^p$, is excellent because it is isomorphic to $R$ via the Frobenius map}. Since $K^p\hookrightarrow K$ is finite by assumption, the integral closure $S$ of $R^p$ in $K$ is a finitely generated $R^p$-module by Proposition \ref{Japanese}. 
But clearly every element of $R$ is integral over $R^p$, as each $r\in R$ satisfies the integral polynomial $x^p-r^p$ over $R^p$. This means that $R$ is an $R^p$-submodule of the Noetherian $R^p$-module $S$, hence  $R$ itself is a Noetherian $R^p$-module. That is, $R$ is finitely generated as an $R^p$-module, and the Frobenius map is finite.  
In other words, $R$ is F-finite.
\end{proof}

\begin{corollary}
\label{Excellent-finite-reduced}
Let $R$ be a reduced, Noetherian ring of characteristic $p$ whose total quotient ring $K$ is F-finite. Then $R$ is excellent if and only if $R$ is F-finite.
\end{corollary}
\begin{proof}
The backward implication is again a consequence of Kunz's Theorem \ref{Kunz-excellent}. So assume that $R$ is excellent. Let $q_1, \dots, q_n$ be the minimal primes of $R$. We denote the corresponding minimal primes of $R^p$ by $q^p_i$. Let $K_i$ be the fraction field of $R/q_i$, so that $K^p_i$ is the fraction field of $R^p/q^p_i$. Then we have a commutative diagram

$$
\xymatrix{
R \ar@{^{(}->}[r] &{R/q_1 \times \dots \times R/q_n} \ar@{^{(}->}[r] & {K_1 \times \dots \times K_n \cong K}\\
R^p \ar@{^{(}->}[u]\ar@{^{(}->}[r] &{R^p/q^p_1 \times \dots \times R^p/q^p_n}\ar@{^{(}->}[u] \ar@{^{(}->}[r] &{K^p_1 \times \dots \times K^p_n \cong K^p} \ar@{^{(}->}[u]}
$$
where all rings involved are $R^p$-modules, and the horizontal maps are injections because $R$ is reduced. Since $R$ is excellent, so is each quotient $R/q_i$, and F-finiteness of $K$ implies that each $K_i$ is also a finitely generated $K^p_i$-module. Thus, Theorem \ref{ExcF-finite} implies that each $R/q_i$ is F-finite, that is, $R/q_i$ a finitely generated $(R/q_i)^p = R^p/q^p_i$-module. As a consequence,
$$R^p/q^p_1 \times \dots \times R^p/q^p_n \hookrightarrow R/q_1 \times \dots \times R/q_n$$
is a finite map, and so is the map $R^p \hookrightarrow  R^p/q^p_1 \times \dots \times R^p/q^p_n$. This shows that $R/q_1 \times \dots \times R/q_n$ is a finitely generated $R^p$-module, and being a submodule of the Noetherian $R^p$-module $R/q_1 \times \dots \times R/q_n$, $R$ is also a finitely generated $R^p$-module. Thus, $R$ is F-finite.
\end{proof}

\medskip

Theorem \ref{ExcF-finite} offers  a simple way to think about
excellence in prime characteristic, at least for domains in function fields. In Section \ref{examplesnon-excellent}, we use Theorem \ref{ExcF-finite} to easily construct many nice examples of non-excellent rings.

{In the spirit of Theorem \ref{ExcF-finite}, there is also an equivalence of excellence and F-finiteness in a slightly different context:}

\begin{theorem}
\cite[Corollary 2.6]{Kunz76}
\label{Local equiv of ExcF-finite}
Let $(R, \mathfrak m)$ be a Noetherian local ring with F-finite residue field. Then $R$ is excellent if and only if $R$ is F-finite.
\end{theorem}

It is worth pointing out that Theorem \ref{Local equiv of ExcF-finite} is closely related to Theorem \ref{ExcF-finite}. {Indeed the backward implication follows from Theorem \ref{ExcF-finite}. Moreover, the hypothesis of Theorem \ref{Local equiv of ExcF-finite} ensures that the completion $\widehat{R}$ is $F$-finite, because Cohen's structure theorem shows that a complete Noetherian local ring  of equal characteristic $ p > 0 $ is F-finite if and only if the residue field is F-finite. The new implication in Theorem \ref{Local equiv of ExcF-finite} then says that if $R$ is excellent, that is, when the completion map $R \rightarrow \widehat{R}$ is well-behaved, $F$-finiteness descends from $\widehat{R}$ to $R$.}


\subsection{Frobenius splitting vs. F-purity} The hypothesis of F-finiteness is often seen in contexts where 
Frobenius splitting is discussed. We recall the definitions of Frobenius splitting and the closely related notion of F-purity, which is  sometimes confused with it. These notions were originally defined in \cite{MehRam85} and \cite{HR1}, respectively.

\begin{definition} 
\label{basic char p notions} 
Let $R$ be an arbitrary commutative ring of prime characteristic $p$.
\begin{enumerate}
\item[(a)] The ring $R$ is \textbf{Frobenius split} if the map $R \overset{F}\rightarrow F_*R$ splits as a map of $R$-modules, that is, there exists an $R$-module map $F_*R \rightarrow R$ such that the composition
$$R \xrightarrow{F} F_*R \rightarrow R$$
is the identity map.
\item[(b)] The ring $R$ is \textbf{F-pure} if $R \overset{F}\rightarrow F_*R$ is a pure map of $R$-modules; this means that the map remains injective after tensoring with any $R$-module $M$.
\end{enumerate}
\end{definition}

It is easy to see  that Frobenius split rings are always F-pure. It is also well-known  that in the presence of F-finiteness, a Noetherian ring is 
Frobenius split if and only if it is F-pure  \cite[Corollary 5.2]{HR1}. However, the relationship between F-purity and Frobenius splitting for a general excellent ring is less understood. Corollary \ref{Excellent-finite-reduced} clarifies that, at least in a large and important setting, 
 there is little difference between the F-finite and excellent settings {for the question of comparing splitting versus purity}:

\begin{corollary}\label{equivInExcellent} For an excellent  Noetherian reduced ring whose total quotient ring is F-finite,  Frobenius splitting is equivalent to F-purity.
 For an excellent  local Noetherian ring  whose  residue field F-finite, Frobenius splitting is equivalent to F-purity.
\end{corollary}

\begin{proof}[Proof of Corollary] It easily follows from the definitions that a split map is pure, so
Frobenius splitting always implies F-purity. Our hypotheses in both statements imply F-finiteness (from  Corollary \ref{Excellent-finite-reduced} and Theorem \ref{Local equiv of ExcF-finite}, respectively), so splitting and purity are equivalent by  \cite[Corollary 5.2]{HR1}.
\end{proof}

\begin{remark}
We do not know any example of an {\it  excellent } F-pure ring that is not F-split. As we see in Section \ref{examplesnon-excellent}, there are plenty of non-excellent examples {even among regular local rings}.
\end{remark}

\section{Maps inverse to Frobenius}
Test ideals are an important  technical tool in both  commutative algebra and birational geometry. The original test ideal of Hochster and Huneke is the ideal generated by all the {\it test elements} for tight closure; they show such test elements exist for excellent local rings in \cite{HochHun94}. Many recent authors have taken the point of view that a slightly  smaller ideal, sometimes called the non-finitistic test ideal, is the more natural object; this ideal is known to be the same as Hochster and Huneke's test ideal in many cases and conjectured to be the same quite generally. See the surveys    \cite{SchweTuck12} or  \cite{SmiZha15} for more information on this history.
    
An important insight of Schwede is that (under appropriate hypothesis) the test ideal can be defined independently of tight closure. 
\begin{definition}
Fix an F-finite ring $R$. An ideal $J$ is said to be uniformly $F$-compatible if for all $e$ and all $\phi \in {\text{Hom}}_R(F_*^eR, R)$, we have $\phi(F^e_*(J))\subset J$.
\end{definition}
     
It is not at all obvious that non-trivial uniformly $F$-compatible ideals exist. Schwede shows, however, using 
a deep theorem of Hochster and Huneke \cite[Theorem 5.10]{HochHun94}, that there is in fact a unique smallest non-zero such ideal \cite{Schw11}. This is the (non-finitistic) {\bf test ideal}.
    
The point we want to emphasize is that the modules $ {\text{Hom}}_R(F_*^eR, R)$ play a crucial role in this approach to test ideals.  Note also that a splitting of Frobenius is a particular element of $ {\text{Hom}}_R(F_*^1R, R)$, namely a map  $F_*R\overset{\phi}\longrightarrow R$ satisfying $\phi(1) = 1$. 
 
Our next theorem shows, however,  that there is little hope to use this approach beyond the F-finite case.
  
\begin{theorem}\label{F-finiteExcellent} Let $R$  be a Noetherian  domain of characteristic $p$  whose fraction
field is $F$-finite. Then the following are equivalent:
\begin{enumerate}
\item 
$R$ is excellent.
\item The Frobenius  endomorphism  $R\overset{F} \longrightarrow F_* R $ is finite.
\item The module ${\operatorname{Hom}}_R(F_*R, R)$  is non-zero.
\item  For all  $e > 0$, the module ${\operatorname{Hom}}_R(F^e_*R, R)$  is non-zero.
\item There  exists $e > 0$  such that  ${\operatorname{Hom}}_R(F^e_*R, R)$  is not the trivial module.
\end{enumerate}
\end{theorem}

Conditions (3)-(5) in Theorem \ref{F-finiteExcellent} can be stated using Hochster's notion of a {\bf solid algebra}. 
 
\begin{definition} 
An $R$-algebra $A$ is {\bf solid} if there exists a non-trivial $R$-module map $A\rightarrow R$.
\end{definition}
 
Thus condition (3) above precisely states that $F_*R$ is a solid $R$-algebra via Frobenius, or equivalently,  that $R$ is a solid $R^p$-algebra. Similarly conditions (4) and (5) deal with the solidity of $R$ over $R^{p^e}$. The theorem states that if $R$ is a  domain whose fraction field is F-finite, then $R$ is a solid algebra via Frobenius if and only if $R$ is excellent.

\begin{remark}
It is worth emphasizing that the generic F-finite assumption in Theorem \ref{F-finiteExcellent} is  essential. 
Fix a field $k$ of characteristic $p$ such that $ [k : k^p ] = \infty$. 
Then   $R = k[x_1, . . . , x_n] $ is an excellent domain that
is not F-finite; in this case $F_*^eR$ is a free $R$-module so  there are many non-zero  maps in $
Hom_R(F^e_*R, R)$.
\end{remark}

\begin{remark}
There are many applications of the module  $ {\text{Hom}}_R(F_*^eR, R)$ which  motivate its study more generally.
Schwede was the first to apply it to the test ideal  in \cite{Schw09} and \cite{Schw10}, but the $R$-module $ {\text{Hom}}_R(F_*^eR, R)$  plays a role in many related stories in birational geometry in characteristic $p$. For example, under suitable hypothesis including F-finiteness, the module ${\text{Hom}}_R(F_*^eR, R)$ can be identified with the global sections of the sheaf $F_*^e\mathcal O_X( (1 - p^e)K_X )$ on $X = $ Spec $R$. Each section of this sheaf can be identified with a $\mathbb Q$-divisor $\Delta$ on Spec $R$ such that $K_X+\Delta$ is $\mathbb Q$-Cartier. This idea is applied to understanding log-Fano varieties in prime characteristic in \cite{SchwSmit10}. Many  other applications are described in \cite{SchweTuck12} and \cite{BlickSchw13}.

It is also worth pointing out for the experts in tight closure that for a local F-finite ring $R$, the uniformly F-compatible ideals defined in terms of the module $ {\text{Hom}}_R(F_*^eR, R)$ can be interpreted as dual to the $\mathcal F(E)$-submodules of $E$ (where $E$ is an injective hull of the residue field and $\mathcal F(E)$ is the ring of all Frobenius operators on it) studied in \cite{LyuSmi01}, the largest of which is the tight closure of zero. The dual characterization used by Schwede to define the test ideal was first carried out in the Gorenstein case in \cite{Smi95} and \cite{Smi97}. 
\end{remark}

The proof of Theorem \ref{F-finiteExcellent} requires the following lemma, which is independent of the characteristic.

\begin{lemma}\label{lem1}
 Let  $R \overset{f}\rightarrow S$  be an injective ring homomorphism of Noetherian domains
such that the induced map of fraction fields $\operatorname{Frac}(R) \hookrightarrow \operatorname{Frac}(S)$ is finite. If the canonical
map
$$S \rightarrow {\operatorname{Hom}}_R({\operatorname{Hom}}_R(S,R),R)$$
is injective, then f is also a finite map.
 \end{lemma}
 
 \begin{proof}
Note that if $M$ is a finitely generated $R$-module, then also  ${\text{Hom}}_R(M,R)$  is finitely generated. Thus, the
lemma follows by Noetherianity if we can show that
${\text{Hom}}_R(S,R)$
is a finitely generated $R$-module.
Let $n$ be the degree of the field extension $\operatorname{Frac}(S)/\operatorname{Frac}(R). $ Then there exists a basis
$x_1, \dots, x_n$ of $ \operatorname{Frac}(S) $ over $\operatorname{Frac}(R)$ such that $x_i \in S$ \cite[5.1.7]{AM69}.

Let $ T$ be the free $ R$-submodule of $S$
generated by the $x_i$. It is clear that $S/T$ is a torsion $R$-module. Then applying  ${\text{Hom}}_R(-,R)$
to the short exact sequence
$$0 \rightarrow T \rightarrow S \rightarrow S/T \rightarrow 0 $$
we get the exact sequence
 $$0 \rightarrow {\text{Hom}}_R(S/T, R) \rightarrow {\text{Hom}}_R(S, R) \rightarrow {\text{Hom}}_R(T, R).$$
Since $S/T$ is a torsion $ R$-module and $R$ is a domain, ${\text{Hom}}_R(S/T,R) = 0. $ Thus, ${\text{Hom}}_R(S,R)$
is a submodule of  ${\text{Hom}}_R(T,R),$ which is free of rank $n$. But $R$ is a Noetherian ring, and so ${\text{Hom}}_R(S,R)$ is also finitely generated.
\end{proof}

A necessary condition for the injectivity of  $S \rightarrow  \operatorname{Hom}_R(\operatorname{Hom}_R(S,R),R) $ in the situation
of the previous lemma is for the module $\operatorname{Hom}_R(S,R)$ to be non-trivial.  
If only the non-triviality of this module is assumed, injectivity of $S \rightarrow  \operatorname{Hom}_R(\operatorname{Hom}_R(S,R),R) $ follows
for a large class of examples as shown in the following result:

\begin{proposition}
\label{prop2}
{Let $ R \overset{f}\hookrightarrow S$ be an injective ring homomorphism of arbitrary domains such that the induced map $\operatorname{Frac}(R) \hookrightarrow  \operatorname{Frac}(S)$ is algebraic. If $S$ is a solid $R$-algebra, then the canonical map $S \rightarrow \operatorname{Hom}_R(\operatorname{Hom}_R(S,R),R)$ is injective. If, in addition, $R$ and $S$  are Noetherian and $f$ is generically finite, then $f$ is a finite map.}

\end{proposition}

\begin{proof}
{
By non-triviality of $\operatorname{Hom}_R(S,R)$, there exists an $R$-linear map $S\overset{\phi} \rightarrow R$ such that $\phi(1) \neq  0$, and so, for all non-zero $r \in R$, $\phi(r) = r\phi(1) \neq 0$. For the injectivity of 
$$S \rightarrow  \operatorname{Hom}_R(\operatorname{Hom}_R(S,R),R),$$
it suffices to show that for each non-zero  $s \in S$, there exists $\varphi \in \operatorname{Hom}_R(S, R)$ such that $\varphi(s)\neq 0$. Now since $x$ is algebraic over $\operatorname{Frac}(R)$, there exists $\sum_{i = 0}^n a_iT^i \in R[T]$ such that $a_0 \neq 0$, and 
$$a_ns^n + a_{n-1}s^{n-1} + \dots a_1s + a_0 = (a_ns^{n-1} + a_{n-1}s^{n-2} + \dots + a_1)s + a_0 =  0.$$ 
Suppose $\ell_{\lambda}$ is left multiplication by $\lambda$, where $\lambda := a_ns^{n-1} + a_{n-1}s^{n-2} + \dots + a_1 \in S$. Then $\phi \circ \ell_{\lambda} \in \operatorname{Hom}_R(S, R)$, and
$$\phi \circ \ell_{\lambda}(s) = \phi (-a_0) = -a_0\phi(1) \neq 0,$$
which proves injectivity of $S \rightarrow  \operatorname{Hom}_R(\operatorname{Hom}_R(S,R),R)$.}

{If $R \overset{f} \rightarrow S$ is a generically finite map of Noetherian domains, then $f$ is a finite map by Lemma \ref{lem1} and what we just proved.}
\end{proof}

\begin{remark}
As a special case of Proposition \ref{prop2}, we obtain the following result: Let $R$ be any domain and $K$ be any field containing $R$. If the integral closure $\overline{R}$ of $R$ in $K$ is a solid $R$-algebra, then the canonical map $\overline{R} \rightarrow \operatorname{Hom}_R(\operatorname{Hom}_R(\overline{R}, R), R)$ is injective. In particular, a Noetherian domain $R$ is Japanese precisely when the integral closure of $R$ in any finite extension of its fraction field is a solid $R$-algebra.
\end{remark}


\begin{proof}[Proof of Theorem \ref{F-finiteExcellent}]
We already know (1) and (2) are equivalent from Theorem \ref{ExcF-finite}. 

{For (2) implies (3), assume $F_*R$ is a finitely generated $R$-module. Let $K$ be the fraction field of $R$, and denote by $F_*K$ the fraction field of $F_*R$, again emphasizing the K-vector space structure via Frobenius. Note $F_*K = F_*R \otimes_R K$. Since
$$\operatorname{Hom}_R(F_*R, R) \otimes_R K \cong \operatorname{Hom}_K(F_*K, K) \neq 0,$$
it follows that $\operatorname{Hom}_R(F_*R, R) \neq 0$.}


We now show (3) implies (4). If $\operatorname{Hom}_R(F_*R, R)$ is non-trivial, then there exists $\phi: F_*R \rightarrow R$ such that
$$\phi(1) = c \neq 0.$$
By induction, suppose there exists $\varphi \in \operatorname{Hom}_R(F^{e-1}_*R, R)$ such that $\varphi(1) \neq 0$. Then the $p^{-e}$-linear map
$$F^e_*R \xrightarrow{F^{e-1}_*(\phi)} F^{e-1}_*R \xrightarrow{\varphi} R$$
maps $c^{(p^{e-1} -1)p} \mapsto c\varphi(1) \neq 0$, showing that $\operatorname{Hom}_R(F^e_*R, R)$ is non-trivial.

 Obviously, (4) implies (5). 
We finish the proof by proving that (5) implies  (2).  By assumption, $F^e_*K$  is a finite extension of K. We now apply Proposition \ref{prop2}, taking 
 taking $S = F^e_*R$ and $f = F^e$. The proposition implies that  $F^e$ is a finite map. Thus, also $F$ is a finite map, and we have proved (5) implies (2).
\end{proof}

\begin{corollary}
If $R$ is a non-excellent domain of characteristic $p > 0$ which is generically $F$-finite, then  $Hom(F_*^eR, R) = 0$ for all $e \in \mathbb{N}$.
\end{corollary}

Basically, this corollary means that we can not expect to develop a theory of test ideals for non-excellent rings, at least, not a theory that uses the ideas of uniform $F$-compatibility.

\section{Examples of non-excellent rings}
\label{examplesnon-excellent}

Given that the class of excellent rings is so large,  it is natural to wonder how one can possibly find  natural classes of examples of non-excellent rings.
The next theorem gives one  source.

\begin{theorem}\label{exDivisorial}
Let $K$ be a field of  characteristic $p$ such that $[K:K^p] < \infty.$  For any discrete valuation ring $V$ of $K$, the following are equivalent:
\begin{enumerate}
\item $V$ is excellent;
 \item $V$ is F-finite;
 \item $V$ is Frobenius split.
\end{enumerate}

{Moreover, if $K$ is a function field over a ground field $k$, and $V$ is a discrete valuation ring of $K/k$, then (1)-(3) are equivalent to $V$ being a divisorial valuation ring of $K/k$.}
\end{theorem}

Recall that a {\bf divisorial valuation ring}  of $K$ is one that obtained as the local ring along some  prime divisor of a normal model of $K/k$. In particular, if  $K/k$ is a function field of transcendence degree $d$ over $k$, then any divisorial valuation ring of $K$ has residue field  of transcendence degree $d-1$ over $k$.

\begin{proof}
{The equivalence of (1) and (2) is a straightforward consequence of Theorem \ref{ExcF-finite}. For the proof of (2) $\Rightarrow$ (3), we use the fact that Frobenius is flat for a regular local ring \cite[Theorem 2.1]{Kunz69}. In particular, when $V$ is F-finite, flatness implies that $F_*V$ is a free $V$-module, which gives a splitting of the Frobenius map. Conversely, a splitting of Frobenius gives the non-triviality of $\operatorname{Hom}_V(F_*V, V)$. Then $V$ is F-finite by Theorem \ref{F-finiteExcellent}.}

{Finally, (1)-(3) is equivalent to $V$ being divisorial when it is a discrete valuation ring of a function field $K/k$ by \cite[Corollary 6.6.3]{DatSmi16}.} 
\end{proof}

\begin{remark}
The paper \cite{DatSmi16},  with corrections in \cite{DatSmi17},
shows more generally that a (not necessarily Noetherian) valuation  ring with F-finite function field will {\it always} be divisorial if it is  F-finite; see \cite[Thm 0.1]{DatSmi17}.
Thus, in the class of valuation rings of an F-finite  function field, F-finiteness  implies Noetherian. 
\end{remark}

\subsection{Some Non-excellent DVRs}
\label{examples} Let $k$ be the algebraic closure of the finite field $\mathbb F_p,$ and fix 
 $K = k(x, y)$,  the function field of $\mathbb P^2_{k}$.  For concreteness, we consider discrete valuations of $K$ centered at the origin (the point defined by the ideal $(x,y)$). The reader will immediately observe that our technique generalizes to any function field over $k$.

Choose any non-unit power series  $p(t) \in k[[t]]$ which is {\it not algebraic } over the subfield  $ k(t) $ of $k((t))$.
Note that such $p(t)$ are abundant: the field $k(t)$ is  countable, hence so its algebraic closure, whereas  $k[[t]]$ is uncountable (consisting of all infinite sequences of elements in $k$).
Thus, there are uncountably many different choices of non-unit power series $p(t)$ non-algebraic over $k(t)$.

Consider the ring homomorphism
$$
k[x, y]\hookrightarrow   k[[t]]
$$
obtained by sending $x\mapsto t$ and $y \mapsto p(t)$. Our assumption on  $p(t)$ ensures this map is injective.
Consider the induced inclusion of fraction fields
$$
 k(x, y)\hookrightarrow k((t)).
$$
The standard $t$-adic valuation on $ k((t))$ restricts to some discrete valuation on $  k(x, y)$  which takes the value 1 on $x$. Its 
valuation ring is 
$V_p =  k [[t]] \cap  k(x, y)$, whose maximal ideal is generated by any element of minimal non-zero value, such as $x$.
We have a local map of local rings
$$
V_p \hookrightarrow   k [[t]] 
$$ 
in which the maximal ideal of $ k [[t]] $ obviously contracts to the maximal ideal of $V_p.$  In particular, the residue field of $V_p$ satisfies
$$ k \hookrightarrow V_p/\frak{m}_{V_p}  \hookrightarrow k [[t]]/(t) \cong k.
$$
Hence, the residue field of $V_p$ is $k$, which has transcendence degree zero over $k$. This means that the {discrete} valuation ring $V_p$ {of $k(x, y)/k$} is {\it{not divisorial}}. 


Moreover, by Theorem \ref{exDivisorial}, because $V_p$ is not divisorial, it is neither excellent nor F-finite.   This gives us examples of non-excellent regular local rings of dimension 1, whose fraction field is $k(x,y)$.

\begin{remark}
Since one can similarly embed $k(x_1, \dots, x_n)$ in $k((t))$ for any $n \geq 2$, our method easily generalizes to produce examples of non-divisorial DVRs in the function field of $\mathbb P^n_k$, for all $n \geq 2$ and $k$ of characteristic $p$. These can be extended to non-divisorial DVRs on the function field of any variety over $k$.
\end{remark}

\medskip{The above} construction shows that there are many more non-excellent DVRs than excellent ones. For example, among DVRs of $\mathbb P^2_k$, we have: 

\begin{corollary}
Let $K= k(x, y)$, where $k $ is the algebraic closure of $\mathbb F_p$.  The  set of all discrete valuation rings of $K/k$ is an uncountable set, with the excellent ones among them forming a countable subset.
\end{corollary}

\begin{proof}
We first show that our construction above already  gives  {\it uncountably many}  non-excellent valuation rings in $k(x,y)$ over $k$. We have already observed that there are uncountably many different choices of the power series $p(t)$ {giving a homorphism of fields
$$k(x, y) \hookrightarrow k((t))$$
that maps $x \mapsto t$ and $y \mapsto p(t)$. Each such homomorphism then gives a discrete valuation
\begin{equation}
\label{def of trans. valuation}
v_{p(t)}: k(x,y)^{\times} \hookrightarrow k((t))^{\times} \xrightarrow{t- adic} \mathbb Z,
\end{equation}
whose associated valuation ring is a non-excellent discrete valuation ring. We now claim that each choice of $p(t)$ yields a {\it different} valuation ring of $k(x, y)/k$.

Let $p(t) = \Sigma_{n \geq 0} a_nt^n$ and $q(t) = \Sigma_{n \geq 0} b_nt^n$ be two different power series, and let $i \in \mathbb N \cup \{0\}$ be the smallest integer such that $a_i \neq b_i$. From the definitions of $v_{p(t)}$ and $v_{q(t)}$ (see \ref{def of trans. valuation}) we get
$$v_{p(t)}(y - (a_0 + a_1x + \dots +  a_ix^i)) > i \hspace{2mm} \operatorname{and} \hspace{2mm} v_{q(t)}(y - (a_0 + a_1x + \dots + a_ix^i)) = i.$$
Thus the fraction 
$$
\frac{x^i}{y-(a_0 + a_1x + \dots + a_ix^i)}
$$ is in the valuation ring for the valuation $v_{q(t)}$ but not  the one for 
$v_{p(t)}$, showing that each choice of power series $p(t)$ gives rise to a distinct valuation ring of 
$k(x, y)$.  This completes the proof that 
$k(x,y)/k$ has uncountably many non-excellent DVRs. } 

{On the other hand, let us show more generally that for any countable algebraically closed field $k$ of characteristic $p$, and any function field $K$ of $k$, the set  divisorial valuation rings of $K/k$ is countable. Note that any such valuation ring is the localization of a finitely generated, normal $k$-subalgebra $R$ of $K$ at a height $1$ prime. Observe that being a finitely generated field extension of a countable field, $K$ itself is also countable.  Thus the collection of all finitely generated $k$-subalgebras $R$ of $K$ is countable.
Any such $R$ is itself countable, and since every ideal of $R$ is finitely generated,  the set of ideals of $R$ is countable. This clearly implies countability of the collection $S$ of pairs $(R, p)$, where $R$ is a finitely generated, normal $k$-subalgebra of $K$ with fraction field $K$ and height one prime $p$, completing the proof.}
\end{proof}

To summarize: randomly choosing a discrete valuation ring in  $k(x,y)/k$, we expect it to be {\it non-excellent} since there are only countably many excellent valuation rings. 
Equivalently, there are only countably many F-finite {discrete} valuation rings in $k(x,y)/k$, namely the same ones which are excellent.

\begin{remark}
See also \cite[Chapter VI]{ZarSam}, \cite{Ben73} and \cite[Example 8.2.31]{Liu06} where these types of rings are discussed. In particular, \cite{Liu06} gives a different argument for the failure of the Japanese property
in a specific case.
\end{remark}


\bibliographystyle{amsplain}
\bibliography{BibliographyEinNov19}

\end{document}